\documentclass[12pt]{article}
\usepackage{amsmath,amsthm,amscd,amssymb}
\usepackage[mathscr]{eucal}
\usepackage{amssymb}
\usepackage{latexsym}

\newtheorem{Def}{Definition}[section]
\newtheorem{Thm}[Def]{Theorem}
\newtheorem{Prop}[Def]{Proposition}
\newtheorem{Rem}[Def]{Remark}

\newtheorem{Lem}[Def]{Lemma}

\numberwithin{equation}{section}
\newcommand{\Q}{\mathbb{Q}}

\newcommand{\Z}{\mathbb{Z}}

\newcommand{\hh}{\mathbb{H}}

\newcommand{\mat}[4]{\begin{pmatrix} #1 & #2 \\ #3 & #4 \end{pmatrix}}
\newcommand{\smat}[4]{\left(\begin{smallmatrix} #1 & #2 \\ #3 & #4 \end{smallmatrix}\right)}

\begin{document}

\title{On $p$-adic Siegel--Eisenstein series II: How to avoid the regularity condition for $p$}
\author{Siegfried B\"ocherer and Toshiyuki Kikuta$^*$}
\maketitle


\begin{abstract}
In a previous paper, the authors showed that two kinds of $p$-adic Siegel--Eisenstein series of degree $n$ coincide with classical modular forms of weight $k$ for $\Gamma _0(p)$, 
under the assumption that $p$ is a regular prime.  
The purpose of this paper is to show that this condition on $p$ can be removed if 
the degree $n$ is low compared with $k$, namely, $n\le 2k+1$.
\end{abstract}

\section{Introduction}

Nagaoka \cite{Na} and Katsurada--Nagaoka \cite{Kat-Na} considered $p$-adic limits of Siegel
Eisenstein series for arbitrary degree $n$, starting from ``initial weight'' $k$
being $1$ and $2$ respectively. 
They showed that these limits become classical Siegel modular forms for $\Gamma_0(p)$,
more precisely they were identified with certain linear combinations of genus theta series.
It was later noticed \cite{Kat-Nacorr} that in their work a
condition ``$p$ regular'' had to be imposed if $n$ is large, but the result remained valid unconditionally
for small $n$.


Using a different technique (mod $p^m$ singular modular
forms and their structure) we recently proved such a
result for arbitrary $k$ for regular primes $p$. 
We emphasize that this condition on $p$ was needed for all $n$ (including the case of small $n$). 
The purpose of the present note is to show that by using some
ingredients from \cite{Bo-Ki1,Bo-Ki2} (but not the full strength of the structure theorem there)
we can also get rid of the $p$-regularity condition if $n$ is not too large compared with $k$ (i.e., $n\le 2k+1$).

\section{Preliminaries}
\label{Sec:2}
\subsection{Siegel modular forms}
\label{sec:siegel-modular-forms}
Let $n$ be a positive integer,
$\hh_{n}$ the Siegel upper half space of degree $n$,
and $\Gamma _n$ the Siegel modular group of degree $n$. 
Let $N$ be a positive integer and $\Gamma _0^{(n)}(N)$ the congruence subgroup of $\Gamma _n$ defined as 
\begin{align*}
&\Gamma _0^{(n)}(N):=\left\{ \begin{pmatrix}A & B \\ C & D \end{pmatrix}\in \Gamma _n \: \Big| \: C\equiv 0_n \bmod{N} \right\}.
\end{align*}

For a positive integer $k$ and a Dirichlet character 
$\chi $ mod $N$, we denote by $M_k(\Gamma _0^{(n)}(N),\chi )$
the space of Siegel modular forms of weight $k$ with character $\chi$ for $\Gamma _0^{(n)}(N)$. 
When $\chi = {\boldsymbol 1}_N$ (trivial character mod $N$), 
we write simply $M_k(\Gamma _0^{(n)}(N))$ for $M_k(\Gamma _0^{(n)}(N),{\boldsymbol 1}_N)$.

Any $F \in M_k(\Gamma _0^{(n)}(N), \chi )$ has a Fourier expansion of the form
\[
F(Z)=\sum_{0\leq T\in\Lambda_n}a_F(T)q^T,\quad q^T:=e^{2\pi i {\rm tr}(TZ)},
\quad Z\in\mathbb{H}_n,
\]
where
\[
\Lambda_n
:=\{ T=(t_{ij})\in {\rm Sym}_n(\mathbb{Q})\;|\; t_{ii},\;2t_{ij}\in\mathbb{Z}\}.
\]

For a subring $R$ of $\mathbb{C}$, we denote by $M_{k}(\Gamma _0^{(n)}(N), \chi )_{R}$ 
the $R$-module consisting of all $F\in M_{k}(\Gamma _0^{(n)}(N), \chi )$ satisfying $a_F(T)\in R$ for all $T\in \Lambda _n$. 
We write simply $M_{k}(\Gamma _0^{(n)}(N))_{R}$ for  
$M_{k}(\Gamma _0^{(n)}(N), {\boldsymbol 1}_N)_{R}$. 
 
We denote by $\Lambda _n^+$ the set of all positive definite elements of $\Lambda _n$. 
Let $a_F(T)^*$ be the rank $r$ primitive Fourier coefficient for $T\in \Lambda _r^+$ ($r\le n$), 
as introduced in \cite{Bo-Ra}, it is characterized by the relation
\begin{align*}
a_F\begin{pmatrix}T & 0 \\ 0 & 0_{n-r} \end{pmatrix}
=\sum _{D}a_F\begin{pmatrix}T[D^{-1}] & 0 \\ 0 & 0_{n-r}\end{pmatrix}^*.  
\end{align*}
Here $D$ runs over all elements of ${\rm GL}_r(\Z)\backslash \{D\in \Z^{r,r}\; |\; \det D\neq 0\}$
satisfying $T[D^{-1}]\in \Lambda _r^+$. 
We define 
\[F_{[r]}:=\sum _{\substack{0\le T\in \Lambda _n \\ {\rm rank}(T)=r}} a_F(T)q^T. \]
The following property will be needed later.  
\begin{Lem}[\cite{Bo-Ki2} Lemma 4.1]
We have (as a formal identity of the Fourier expansions)
\label{lem:Fri}
\begin{align}
F_{[r]}=\sum _{S\in \Lambda_{r}^{+}/{\rm GL_r}(\Z)}\frac{a_F\left(\begin{smallmatrix}S & 0 \\ 0 & 0_{n-r}\end{smallmatrix}\right)^*}{\epsilon(S)}(\theta ^{(n)}_S)_{[r]},
\end{align}
where $S$ runs over representatives of all ${\rm GL_r}(\Z)$-equivalence classes in $\Lambda_{r}^{+}$, 
$\theta _S^{(n)}$ is the theta series of degree $n$ attached to $S$ defined as 
\begin{align*}
\theta _S^{(n)}(Z):=\sum _{X\in \Z^{r,n}}e^{2\pi i({\rm tr}(S[X]Z))}\quad (Z\in \hh_{n}) 
\end{align*}
($\Z^{r,n}$ is the set of  $r\times n$ matrices with integral components),    
and $\epsilon (S)$ is the cardinality of the group of automorphisms of $S$. 
\end{Lem}

\subsection{Congruences for modular forms}
Let $p$ be an odd prime. 
Let $F_i$ ($i=1$, $2$) be two formal power series of the form
\[F_i=\sum _{T\in \Lambda _{n}}a_{F_i}(T)q^T\]
with $a_{F_i}(T)\in \Z_{(p)}$ for all $T\in \Lambda _n$. 
We write $F_1 \equiv F_2$ mod $p^m$ if $a_{F_1}(T)\equiv a_{F_2}(T)$ mod $p^m$ for all $T \in \Lambda _n$.  

\begin{Def}
We say that $F\in M_k(\Gamma _0^{(n)}(N),\chi )_{\Z_{(p)}}$ is ``mod $p^m$ singular'' if there exists $r$ satisfying the following properties: 
\begin{itemize}  \setlength{\itemsep}{-3pt}
\item
$a_F(T)\equiv 0$ mod $p^m$ for any $T\in \Lambda _{n}$ with ${\rm rank}(T)>r$, 
\item
there exists $T\in \Lambda _n$ with ${\rm rank}(T)=r$ such that $a_F(T)\not \equiv 0$ mod $p$. 
\end{itemize} 
We call such $r$ the ``$p$-rank'' of $F$. 
\end{Def}

\begin{Thm}[\cite{Bo-Ki3}]
\label{thm:wt-sing}
Let $p$ be an odd prime and $k$ a positive integer.  
Let $F\in M_k(\Gamma _0^{(n)}(N),\chi )_{\Z_{(p)}}$ with 
$\chi $ a quadratic Dirichlet character mod $N$. 
Suppose that $F$ is mod $p^m$ singular of $p$-rank $r$. 
Then we have $2k-r\equiv 0$ mod $(p-1)p^{m-1}$.
In particular, $r$ should be even.      
\end{Thm}

\subsection{Some $p$-adic properties of the Siegel--Eisenstein series}
We put 
\[\Gamma _\infty ^{(n)}:=\left\{\mat{A}{B}{C}{D}\in \Gamma _n\; \Big| \; C=0_n\right\}. \]
For an even integer $k$ with $k>n+1$, the Siegel--Eisenstein series of degree $n$ and weight $k$ with level $1$ is defined as 
\[E_k^{(n)}(Z):=\sum _{\smat{*}{*}{C}{D}\in \Gamma _\infty^{(n)}\backslash \Gamma_n}\det (CZ+D)^{-k},\quad Z\in \hh_n.  \]
This series defines an element of $M_k(\Gamma _n)_\Q$. 
We denote by
\[E_k^{(n)}=\sum _{T\in \Lambda _n }a_k^{(n)}(T)q^T\]
the Fourier expansion of $E_k^{(n)}$.

We put $\Lambda _n^{(r)}:=\{T\in \Lambda _n\;|\; {\rm rank}(T)=r\}$. 
Note that, if $T\in \Lambda ^{(r)}_n$ then we have
 $a_k^{(n)}(T)=a_k^{(r)}(T')$ for $T'\in \Lambda _r^+$ such that $T\sim \smat{T'}{0}{0}{0}$ mod ${\rm GL}_n(\Z)$.
Therefore, we may assume that $T\in \Lambda _{r}^+$ and may consider $a_k^{(r)}(T)$, 
instead of $a_k^{(n)}(T)$ with $T\in \Lambda _n^{(r)}$.

Let $p$ be an odd prime. 
We put 
\[{\boldsymbol X}=\Z_p\times \Z/(p-1)\Z. \]
We regard $\Z\subset {\boldsymbol X}$ via the embedding $n\mapsto (n,\widetilde{n})$ with $\widetilde{n}:=n$ mod $p-1$. 
For $(k,a)\in {\boldsymbol X}$, the $p$-adic Siegel--Eisenstein series $\widetilde{E}^{(n)}_{(k,a)}$
of degree $n$ and weight $(k,a)$ is defined as 
\[\widetilde{E}^{(n)}_{(k,a)}:=\lim _{m\to \infty }E_{k_m}^{(n)} \quad (p\text{-adic\ limit}) \]
where $k_m$ is a number sequence such that $k_m\to \infty $ ($m\to \infty$) in the usual topology of $\mathbb{R}$, and $k_m\to (k,a)$ ($m\to \infty$) in ${\boldsymbol X}$.

In this paper, we treat the case when the weights are $(k,k)$ with $k$ even and $(k,k+\frac{p-1}{2})$ 
with condition $k\equiv \frac{p-1}{2}$ mod $2$. 
We put ${\boldsymbol k}_0:=(k,k)$ and ${\boldsymbol k}_1:=(k,k+\frac{p-1}{2})$ and denote by $k_j(m)$ 
a number sequence satisfying $k_j(m)\to {\boldsymbol k}_j$ ($m\to \infty$) in ${\boldsymbol X}$.  
For any $k_j(m)$ satisfying $k_j(m)\to {\boldsymbol k}_j$ ($m\to \infty $) in ${\boldsymbol X}$, 
we may assume that $k_j(m)$ is of the form
\[k_j(m)=k+a_j(m)p^{b(m)}, \]
where $a_j(m)$ is a positive integer with $a_j(m)\equiv \frac{p-1}{2^j}$ mod $p-1$, 
and $b(m)=b_j(m)$ is a positive integer satisfying $b(m)\to \infty $ if $m\to \infty$.  
Thereafter, whenever ${\boldsymbol k}_j$ or $k_j(m)$ is involved, 
the above condition on $k$ is always assumed.

Let $v_p$ be the additive valuation on the $p$-adic field $\Q_p$ normalized such that $v_p(p)=1$. 
For a formal power series $F$ of the form $F=\sum _{T\in \Lambda _{n}}a_{F}(T)q^T$ with $a_{F}(T)\in \Q_p$,  
we define
\begin{align*}
&v_p^{(r)}(F):=\inf \{v_p(a_{F}(T))\; |\; T\in \Lambda _n,\ {\rm rank}(T)=r\}.  
\end{align*}
We put $\nu _m:=v^{(2k)}_p(E^{(2k+1)}_{k_j(m)})$. 
\begin{Prop}[\cite{Bo-Ki1}]
\label{prop:sing}
Let $p>2k+1$. 
\begin{enumerate}\setlength{\itemsep}{-3pt}
\item 
We have $\nu _m \le 0$ and $\nu _m$ is a constant for sufficiently large $m$. 
\item 
We have $p^{-\nu _m}E_{k_j(m)}^{(2k+1)}\in M_{k_j(m)}(\Gamma _n)_{\Z_{(p)}}$ and 
this is a mod $p^{c(m)}$ singular of $p$-rank $2k$. 
Here $c(m)$ is a positive integer satisfying $c(m)\to \infty $ when $m\to \infty $. 
\end{enumerate}
\end{Prop}

\begin{Rem}
Both properties (1), (2) were proved in \cite{Bo-Ki1} Proposition 3.3.  
As written in \cite{Bo-Ki1} Remark 3.4, the regularity condition for $p$ is not necessary when $n=2k+1$. 
\end{Rem}

For $S\in \Lambda _r^+$ with $r$ even, we define $\chi _S$ by 
\begin{align*}
\chi _S(d):={\rm sign} (d)^\frac{r}{2} \left( \frac{(-1)^\frac{r}{2}\det 2S}{|d|} \right).
\end{align*}  
Let $\eta _S$ be the primitive character associated with $\chi _S$. 
Let $\chi _p:=(\frac{*}{p})$ be the unique nontrivial quadratic character mod $p$.  
We write $\chi _p^0$ for the trivial character mod $p$. 
The level of $S$ is defined as the smallest positive integer $l$ such that $l(2S)^{-1}\in 2\Lambda _{r}$. 
We denote it by ${\rm level}(S)$. 
\begin{Lem}
\label{lem:lev-ch}
Let $p>2k+1$ and $S\in \Lambda _{2k}^+$. 
\begin{enumerate}\setlength{\itemsep}{-3pt}
\item
We have $\lim _{m\to \infty }a_{k_j(m)}^{(2k)}(S)^*\in \Q$ ($p$-adic limit). 
\item
If $\lim _{m\to \infty}a_{k_j(m)}^{(2k)}(S)^*\neq 0$ then ${\rm level}(S)\mid p$ and $\chi _S=\chi _p^j$. 
\item
For any $S\in \Lambda _{2k}$ with ${\rm level}(S)\mid p$, we have $v_p(a_{k_j(m)}^{(2k)}(S)^*)\ge \nu _m$
when $m$ is sufficiently large. 
\end{enumerate}
\end{Lem}
\begin{proof}
The assertions (1), (3) follow from \cite {Bo-Ki1} Subsections 3.2, 3.3. 

We prove (2). As written in \cite{Bo-Ki1} Subsection 3.3, 
if $\lim _{m\to \infty }a_{k_j(m)}^{(2k)}(S)\neq 0$ then we have 
$\eta _S=\begin{cases}{\boldsymbol 1} \ (j=0) \\ \chi _p\ (j=1). \end{cases}$ 
Suppose that $\lim _{m\to \infty} a_{k_j(m)}^{(2k)}(S)^*\neq 0$. 
Then ${\rm level}(S)\mid p$ follows from the reasoning of \cite{Bo-Ki1} Subsection 3.3 again. 
We need to confirm $\chi _S=\chi _p^j$. 
Recall that $a_{k_m}^{(2k)}(S)^*$ is a linear combination of $a_{k_m}^{(2k)}(S[D^{-1}])$ with an element $D$ of ${\rm GL}_n(\Z)\backslash \{D\in \Z^{n,n}\; |\; \det D\neq 0\}$ satisfying $T[D^{-1}]\in \Lambda _r^+$. 
There exists at least one $D$ such that $\lim _{m\to \infty } a_{k_j(m)}^{(2k)}(S[D^{-1}])\neq 0$. 
By the above fact, we have $\eta _{S[D^{-1}]}=\begin{cases}{\boldsymbol 1} \ (j=0) \\ \chi _p\ (j=1). \end{cases}$ 
Since $\eta _S=\eta _{S[D^{-1}]}$, we obtain $\eta _S=\begin{cases}{\boldsymbol 1} \ (j=0) \\ \chi _p\ (j=1). \end{cases}$
Therefore we get $\chi _S=\chi _p^j$ because of ${\rm level}(S)\mid p$.  
\end{proof}

\section{Main theorem and its proof}
Let $S\in \Lambda _m^+$ and ${\rm gen}(S)$ be the genus containing $S$. 
Let $\{S_1,\cdots ,S_h\}$ be a set representatives of ${\rm GL}_m(\Z)$-equivalence classes in ${\rm gen}(S)$. 
The genus theta series associated with $S$ is defined by 
\[\Theta ^{(n)}_{{\rm gen}(S)}(Z):=\left(\sum _{i=1}^h\frac{\theta _{S_i}^{(n)}(Z)}{\epsilon(S_i)}\right)/\left(\sum _{i=1}^h\frac{1}{\epsilon (S_i)}\right). \]
We put 
\[(\Theta ^{(n)}_{{\rm gen}(S)})^0:=\left(\sum _{i=1}^h\frac{1}{\epsilon (S_i)}\right)\cdot \Theta ^{(n)}_{{\rm gen}(S)}. \]

The following theorem is the main result of this paper.
\begin{Thm}
\label{thm:main}
Let $n$, $k$ be positive integers with $n\le 2k+1$ and $p$ a prime with $p>2k+1$.
Assume that $k$ is even in ${\boldsymbol k}_0$ and 
$k\equiv \frac{p-1}{2}$ mod $2$ in ${\boldsymbol k}_1$. 
Then we have
\begin{align}
\label{eq:lin0}
\widetilde{E}_{{\boldsymbol k}_j}^{(n)}= \sum _{\substack{{\rm gen}(S)\\ \ {\rm level}(S)\mid p \\ \chi _S=\chi _p^j}}\widetilde{a}_j({\rm gen}(S))\cdot (\Theta ^{(n)}_{{\rm gen}(S)})^0,\quad (\widetilde{a}_j({\rm gen}(S))\in \Q), 
\end{align}
where the summation in (\ref{eq:lin0}) goes over all genera ${\rm gen}(S)$ of $S\in \Lambda _{2k}^+$ with level dividing $p$ satisfying $\chi_S=\chi _p^j$.   

\end{Thm}
\begin{Rem}
\begin{enumerate} \setlength{\itemsep}{-3pt}
\item
More precisely, we have $\widetilde{a}_j({\rm gen}(S))=\lim _{m\to \infty }a_{k_j(m)}^{(2k)}(S)^*$. 
This $p$-adic limit is given explicitly by the formula in \cite{Bo-Ki1} Theorem 3.8. 
\item
The main point of this theorem is that the regularity condition on $p$ we assumed in \cite{Bo-Ki1} Theorem 1.1 (under no condition on $n$ at all) can be removed when $n\le 2k+1$.
\item
By the same proof as \cite{Bo-Ki1} Corollary 1.3, we have 
\[\widetilde{E}_{{\boldsymbol k}_j}^{(n)}\mid U(p)=\widetilde{E}_{{\boldsymbol k}_j}^{(n)}, \]
where $U(p)$ is the standard Hecke operator of level $p$. 
\item
We notice that the condition $p>2k+1$ came up in Proposition \ref{prop:sing} and Lemma \ref{lem:lev-ch}.
It also will appears in the final step of the proof for Theorem \ref{thm:main}. 
We do not know if this condition is really necessary. 
\end{enumerate}
\end{Rem}

\begin{proof}[Proof of Theorem \ref{thm:main}]
It suffices to consider the case $n=2k+1$, since we may apply the Siegel $\Phi$-operator to $\widetilde{E}_{{\boldsymbol  k}_j}^{(2k+1)}$ if $n<2k+1$. 

By Proposition \ref{prop:sing}, $p^{-\nu _m }E_{k_j(m)}^{(2k+1)}$ is a mod $p^{c(m)}$ singular of $p$-rank $2k$ for some $c(m)$ satisfying $c(m)\to \infty $ when $m\to \infty $. 
Applying Lemma \ref{lem:Fri} to $p^{-\nu _m }E_{k_j(m)}^{(2k+1)}$, we have 
\[p^{-\nu _m }(E_{k_j(m)}^{(2k+1)})_{[2k]}=\sum _{S\in \Lambda _{2k}^+/{\rm GL}_{2k}(\Z)}p^{-\nu _m }\frac{a_{k_j(m)}^{(2k)}(S)^*}{\epsilon (S)}
(\theta _S^{(2k+1)})_{[2k]}.\]
Note that $p^{-\nu _m }\frac{a_{k_j(m)}^{(2k)}(S)^*}{\epsilon (S)}\in \Z_{(p)}$, since $v_p(a_{k_j(m)}^{(2k)}(S)^*)\ge \nu _m$ (Proposition \ref{prop:sing} (3)) and $p>2k+1$ implies $p\nmid \epsilon (S)$.  
By Lemma \ref{lem:lev-ch}, we have
\begin{align}
\label{eq:cong}
p^{-\nu _m }(E_{k_j(m)}^{(2k+1)})_{[2k]}\equiv \sum _{\substack{S\in \Lambda _{2k}^+/{\rm GL}_{2k}(\Z)\\ {\rm level}(S)\mid p \\ \chi _S=\chi _p^j}}p^{-\nu _m }\frac{a_{k_j(m)}^{(2k)}(S)^*}{\epsilon (S)}
(\theta _S^{(2k+1)})_{[2k]} \bmod{p^{c'(m)}}  
\end{align}
for some $c'(m)$ satisfying $c'(m)\to \infty $ when $m\to \infty $. 
In particular, this is a finite sum.

We prove that 
\begin{align}
\label{eq:gen2}
p^{-\nu _m }E_{k_j(m)}^{(2k+1)}\equiv \sum _{\substack{S\in \Lambda _{2k}^+/{\rm GL}_{2k}(\Z)\\ {\rm level}(S)\mid p \\ \chi _S=\chi _p^j}}p^{-\nu _m }\frac{a_{k_j(m)}^{(2k)}(S)^*}{\epsilon (S)}
\theta _S^{(2k+1)} \bmod{p^{c''(m)}}
\end{align}
for some $c''(m)$ satisfying $c''(m)\to \infty $ when $m\to \infty $.
Taking ${\mathcal E}\in M_{\frac{p-1}{2}}(\Gamma _0^{(2k+1)}(p),\chi _p)_{\mathbb{Z}_{(p)}}$ from \cite {Bo-Na} such that ${\mathcal E}\equiv 1$ mod $p$, we consider 
\begin{align*}
F:&=p^{-\nu _m }E_{k_j(m)}^{(2k+1)}-\sum _{\substack{S\in \Lambda _{2k}^+/{\rm GL}_{2k}(\Z)\\ {\rm level}(S)\mid p \\ \chi _S=\chi _p^j}}p^{-\nu _m }\frac{a_{k_j(m)}^{(2k)}(S)^*}{\epsilon (S)}
\theta _S^{(2k+1)}\cdot {\mathcal E}^{t_j(m)p^{b(m)}}\\
&\in M_{k_j(m)}(\Gamma _0^{(2k+1)}(p),\chi _p^j)_{\Z_{(p)}}. 
\end{align*}
Here $t_j(m)\in \Z_{\ge 0}$ is defined as the number satisfying $a_j(m)=t_j(m)\cdot \frac{p-1}{2}$.  
(Recall that $k_j(m)=k+a_j(m)\cdot p^{b(m)}$, $b(m)\to \infty $ ($m\to \infty $), and $a_j(m)\equiv \frac{p-1}{2^j} \bmod{p-1}$.) 

We put $c''(m):=\min \{c'(m),b(m)+1\}$. 
Then we want to prove $F\equiv 0$ mod $p^{c''(m)}$. 
Seeking a contradiction, we suppose that $F\not \equiv 0$ mod $p^{c''(m)}$. 
Then $F$ is a mod $p^{c''(m)}$ singular of $p$-rank $r$ with $0\le r<2k$ because of (\ref{eq:cong}). 
Applying Theorem \ref{thm:wt-sing}, we have $2k_j(m)-r\equiv 0$ mod $(p-1)p^{c''(m)-1}$. 
In particular, $2k-r\equiv 0$ mod $p-1$. Since $r<2k<p-1$, this is a contradiction. 
Hence we have $F\equiv 0$ mod $p^{c''(m)}$.

Note that $\nu_m \to \nu $ ($m\to \infty $) for some constant $\nu$ (Proposition \ref{prop:sing} (1)), 
and $\lim _{m\to \infty }a_{k_j(m)}^{(2k)}(S)^*=:\widetilde{a}_j(S)\in \Q$ (Lemma \ref{lem:lev-ch} (1)). 
Taking a $p$-adic limit of both sides of (\ref{eq:gen2}), we have 
\begin{align*}
p^{-\nu }\widetilde{E}_{{\boldsymbol k}_j}^{(n)}= \sum _{\substack{S\in \Lambda _{2k}^+/{\rm GL}_{2k}(\Z)\\ {\rm level}(S)\mid p \\ \chi _S=\chi _p^j}}p^{-\nu }\frac{\widetilde{a}_j(S)}{\epsilon (S)}
\theta _S^{(2k+1)}.
\end{align*}
Then we get (\ref{eq:lin0}) from this formula by resummation, observing that
$a_{k_j(m)}^{(2k)}(S)^*$ (and hence $\widetilde{a}_j(S)$) depends only on the genus of $S$. 
This completes the proof of Theorem \ref{thm:main}. 
\end{proof}

\section*{Acknowledgment}
This work was supported by JSPS KAKENHI Grant Number 22K03259.


\begin{thebibliography}{10}

\bibitem{Bo-Ki3}
S.~B\"ocherer, T.~Kikuta, On mod $p$ singular modular forms. Forum Math. 28 (2016), no. 6, 1051-1065.

\bibitem{Bo-Ki2}
S.~B\"ocherer, T.~Kikuta, Structure theorem for mod $p^m$ singular forms. arXiv:2302.00309 [math.NT]. 

\bibitem{Bo-Ki1}
S.~B\"ocherer, T.~Kikuta, On $p$-adic Siegel--Eisenstein series from a point of view of the theory of mod $p^m$ singular forms. Manuscripta Math. 175(2024), no. 1-2, 203-226. 
 
\bibitem{Bo-Na}
S.~B\"ocherer, S.~Nagaoka, On Siegel modular forms of level $p$ and their properties mod $p$. 
Manuscripta Math. 132 (2010), no. 3-4, 501-515.

\bibitem{Bo-Ra}
S.~B\"ocherer, S.~Raghavan, On Fourier coefficients of Siegel modular forms. J. Reine Angew. Math. 384 (1988), 80-101. 

\bibitem{Kat-Na}
H.~Katsurada, S.~Nagaoka, 
On $p$-adic Siegel Eisenstein series. J. Number Theory 251 (2023), 3-30. 

\bibitem{Kat-Nacorr}
H.~Katsurada, S.~Nagaoka, 
Corrigendum to ``On $p$-adic Siegel Eisenstein series'' [J. Number Theory 251 (2023) 3–30].
J. Number Theory 259 (2024), 419--421.

\bibitem{Na}
S.~Nagaoka, A remark on Serre's example of $p$-adic Eisenstein series. Math. Z. 235 (2000), no. 2, 227-250. 




  
\end{thebibliography}

\section*{Conflict of interest statement}
On behalf of all authors, the corresponding author states that there is no conflict of interest.

\section*{Data availability} 
During the work on this publication, no data sets were generated, used
or analyzed. Thus, there is no need for a link to a data repository.

\providecommand{\bysame}{\leavevmode\hbox to3em{\hrulefill}\thinspace}
\providecommand{\MR}{\relax\ifhmode\unskip\space\fi MR }
\providecommand{\MRhref}[2]{%
  \href{http://www.ams.org/mathscinet-getitem?mr=#1}{#2}
}
\providecommand{\href}[2]{#2}

\begin{flushleft}
Siegfried B\"ocherer\\
Kunzenhof 4B \\
79177 Freiburg, Germany \\
Email: boecherer@t-online.de
\end{flushleft}

\begin{flushleft}
  Toshiyuki Kikuta\\
  Faculty of Information Engineering\\
  Department of Information and Systems Engineering\\
  Fukuoka Institute of Technology\\
  3-30-1 Wajiro-higashi, Higashi-ku, Fukuoka 811-0295, Japan\\
  E-mail: kikuta@fit.ac.jp
\end{flushleft}

\end{document}